\documentclass[11pt]{article}
\usepackage{amssymb,amsthm,amsmath,algorithm2e,float,cite,enumerate}
\usepackage{geometry}

\newtheorem{theorem}{Theorem}[section]
\newtheorem{lemma}[theorem]{Lemma}
\newtheorem{corollary}[theorem]{Corollary}

\newtheorem{claim}{Claim}

\usepackage{setspace}
\usepackage{lineno}

\begin{document}

\onehalfspace
%\linenumbers

\title{Vaccinate your trees!}

\author{Stefan Ehard \and Dieter Rautenbach}

\date{}

\maketitle

\begin{center}
Institut f\"{u}r Optimierung und Operations Research, 
Universit\"{a}t Ulm, Ulm, Germany,
\{\texttt{stefan.ehard,dieter.rautenbach}\}\texttt{@uni-ulm.de}\\[3mm]
\end{center}

\begin{abstract}
For a graph $G$ 
and an integer-valued function $\tau$ on its vertex set,
a dynamic monopoly is a set of vertices of $G$
such that iteratively adding to it vertices $u$ of $G$ 
that have at least $\tau(u)$ neighbors in it
eventually yields the vertex set of $G$.
We study two vaccination problems, 
where the goal is to maximize 
the minimum order of such a dynamic monopoly 
\begin{itemize}
\item either by increasing the threshold value of $b$ vertices beyond their degree,
\item or by removing $b$ vertices from $G$,
\end{itemize}
where $b$ is a given non-negative integer corresponding to a budget.
We show how to solve these problems efficiently for trees.
\end{abstract}

{\small 
\begin{tabular}{lp{13cm}}
\end{tabular}
}

%\pagebreak

\section{Introduction}

Dynamic monopolies are a popular graph-theoretic model for spreading processes.
In a simple yet natural model \cite{keklta,drro,dori}, 
every vertex $u$ of a graph $G$ has a threshold value $\tau(u)$ 
and will be reached by the spreading process 
if at least $\tau(u)$ of its neighbors have been reached.
A set $D$ of vertices is a dynamic monopoly 
if starting the spreading process from the vertices in $D$,
eventually all vertices of $G$ will be reached.
Finding the minimum order ${\rm dyn}(G,\tau)$ of a dynamic monopoly 
is a very hard problem \cite{ch,cedoperasz}.
Some general bounds are known \cite{acbewo,gera,chly} 
but efficient algorithms have only been found 
for restricted instances that essentially possess tree structure \cite{cicogamipeva,ch,cedoperasz,chhuliwuye},
in fact, the tractability of dynamic monopolies appears to be
closely related to the boundedness of the treewidth \cite{behelone}.

The parameter ${\rm dyn}(G,\tau)$ measures a vulnerability of $(G,\tau)$ with respect to pandemic spreading processes.
In the present paper we study two vaccination problems 
corresponding to the reduction of this vulnerability 
subject to a budget constraint.
For a non-negative integer $b$ quantifying our {\it budget},
we want to maximize ${\rm dyn}(G,\tau)$ 
\begin{itemize}
\item[(\ref{ed1})] either by increasing the threshold value of $b$ vertices beyond their degree,
\item[(\ref{ed1b})] or by removing $b$ vertices from $G$.
\end{itemize}
In (\ref{ed1}), the $b$ vertices become immune 
against the infection by the spreading process; 
they can never be reached by the process 
unless they belong to the set from which the spreading starts,
that is, 
every dynamic monopoly for the modified threshold function 
has to contain them.
In (\ref{ed1b}), the $b$ vertices no longer participate in the spreading process at all.

In order to explain our results and discuss related work, we introduce some notation. 
Let $G$ be a finite, simple, and undirected graph.
A {\it threshold function} for $G$ is a function $\tau:U\to \mathbb{Z}\cup \{ \infty\}$ whose domain $U$ contains the vertex set $V(G)$ of $G$.
Let $\tau$ be a threshold function for $G$.
For a set $D$ of vertices of $G$,
the {\it hull $H_{(G,\tau)}(D)$ of $D$ in $(G,\tau)$}
is the smallest set $H$ of vertices of $G$ such that 
$D\subseteq H$, and 
$u\in H$ for every vertex $u$ of $G$ with $|H\cap N_G(u)|\geq \tau (u)$.
Clearly, the set $H_{(G,\tau)}(D)$ is obtained 
by starting with $D$, and 
iteratively adding vertices $u$ 
that have at least $\tau(u)$ neighbors in the current set
as long as possible.
The set $D$ is a {\it dynamic monopoly of $(G,\tau)$}
if $H_{(G,\tau)}(D)$ equals the vertex set of $G$.
Let ${\rm dyn}(G,\tau)$ be the minimum order 
of a dynamic monopoly of $(G,\tau)$.
A dynamic monopoly of $(G,\tau)$ of order ${\rm dyn}(G,\tau)$ is {\it minimum}.
Note that every dynamic monopoly $D$ of $(G,\tau)$ 
necessarily contains each vertex $u$ of $G$ with $d_G(u)+1\leq \tau(u)\leq \infty$,
where $d_G:V(G)\to \mathbb{N}_0$ is the degree function.

To formulate problem (\ref{ed1}), 
we need to manipulate the threshold function $\tau$.
For a set $X$, 
let $\tau_X:U\to \mathbb{Z}\cup \{ \infty\}$ be such that 
$$\tau_X(u)=
\begin{cases}
\tau(u) & \mbox{, if $u\in U\setminus X$, and}\\
\infty & \mbox{, if $u\in U\cap X$.}
\end{cases}$$
Our two vaccination problems (\ref{ed1}) and (\ref{ed1b}) 
can now be written as follows:
\begin{eqnarray}
{\rm vacc}_1(G,\tau,b)&=&\max\left\{ {\rm dyn}(G,\tau_X):X\in {V(G)\choose b}\right\}\mbox{ and }\label{ed1}\\[3mm]
{\rm vacc}_2(G,\tau,b)&=&\max\left\{ {\rm dyn}(G-Y,\tau):Y\in {V(G)\choose b}\right\}\label{ed1b}
\end{eqnarray}
for a given triple $(G,\tau,b)$, where $b$ is a non-negative integer, 
and ${U \choose k}$ denotes the set of all $k$-element subsets of $U$.
Our contribution are efficient algorithms computing
${\rm vacc}_1(T,\tau,b)$
and
${\rm vacc}_2(T,\tau,b)$
for a given triple $(T,\tau,b)$, where $T$ is a tree.

Since (\ref{ed1}) and (\ref{ed1b}) are defined by max-min-expressions, 
where already the inner minimization problem is hard,
it is not surprising that ${\rm vacc}_1(G,\tau,b)$ and ${\rm vacc}_2(G,\tau,b)$ are also hard.
In fact, ${\rm vacc}_1(G,\tau,0)={\rm vacc}_2(G,\tau,0)={\rm dyn}(G,\tau)$,
and all hardness results for ${\rm dyn}(G,\tau)$ 
immediately carry over to these new parameters.
If the order $n(G)$ of $G$ is less than $b$, 
then ${\rm vacc}_1(G,\tau,b)={\rm vacc}_2(G,\tau,b)=\max\emptyset=-\infty$.

Before we proceed to our results, we discuss some related work.

Khoshkhah and Zaker \cite{khza} consider the problem to determine
\begin{eqnarray}\label{ed2}
\max\left\{
{\rm dyn}(G,\tilde{\tau}):
\tilde{\tau}:V(G)\to \mathbb{Z}
\mbox{ such that }
0\leq \tilde{\tau} \leq d_G
\mbox{ and }\sum\limits_{u\in V(G)}\tilde{\tau}(u)\leq b\right\},
\end{eqnarray}
for a given graph $G$ and non-negative integer $b$,
where inequalities between functions are meant pointwise. 
They show the hardness of this problem for planar graphs,
and describe an efficient algorithm for trees.
Centeno and Rautenbach \cite{cera} provide upper bounds on (\ref{ed2}) for general graphs.
\cite{cera,khza} also contain results concerning a variant of (\ref{ed2}), 
where ``$\tilde{\tau}\leq d_G$'' is replaced by ``$\tilde{\tau}\leq d_G+1$'',
and closed formulas are obtained in some cases.
Whereas our problems (\ref{ed1}) and (\ref{ed1b}) model a {\it complete immunization} against infection,
problem (\ref{ed2}) models a {\it partial immunization},
which for $\tilde{\tau}\leq d_G$ can not be complete.
Furthermore, 
whereas we consider a given initial threshold function $\tau$ 
as a lower bound for $\tau_X$ in (\ref{ed1}), 
the problem (\ref{ed2}) uses $0$ as a lower bound for $\tilde{\tau}$,
that is, the corresponding initial threshold would be $0$ everywhere.
Replacing ``$0\leq \tilde{\tau}$'' by ``$\tau\leq \tilde{\tau}$'' within (\ref{ed2}) 
for a given function $\tau$, 
leads to a much harder problem 
to which the methods from \cite{khza,cera} do not seem to apply.

In \cite{bhkllerosh} Bhawalkar, Kleinberg, Lewi, Roughgarden, and Sharma 
study so-called anchored $k$-cores.
For a given graph $G$, and a positive integer $k$, 
the {\it $k$-core} of $G$ is 
the largest induced subgraph of $G$ of minimum degree at least $k$.
It is easy to see that the vertex set of the $k$-core of $G$
equals $V(G)\setminus H_{(G,\tau)}(\emptyset)$
for the special threshold function $\tau=d_G-k+1$.
Now, the {\it anchored $k$-core problem} \cite{bhkllerosh} is to determine 
\begin{eqnarray}\label{ed3}
\max\left\{
\Big|V(G)\setminus H_{(G,\tau_X)}(\emptyset)\Big|:X\in {V(G)\choose b}\right\},
\end{eqnarray}
for a given graph $G$ and non-negative integer $b$.
Bhawalkar et al. show that (\ref{ed3}) is hard to approximate in general,
but can be determined efficiently for $k=2$, and for graphs of bounded treewidth.
Clearly, (\ref{ed3}) bears less similarity with our problems
than (\ref{ed2}). It is defined by a simple max-expression, which makes it easier to handle algorithmically.
Vaccination problems in random settings were studied in \cite{keklta,brjama,de}.

The next section contains our results and their proofs.

\section{Results}

Throughout this section, 
let $T$ be a tree rooted in some vertex $r$, 
and let $\tau:U\to \mathbb{Z}\cup \{ \infty\}$ 
be a threshold function for $T$.
For a vertex $u$ of $T$ and a set $X$, 
let $T_u$ be the subtree of $T$ induced by $u$ and its descendants,
let $\tau^u:U\to \mathbb{Z}\cup \{ \infty\}$ 
be such that 
$$\tau^u(v)=
\begin{cases}
\tau(v) & \mbox{, if $v\in U\setminus \{ u\}$, and}\\
\tau(v)-1 & \mbox{, if $v=u$,}
\end{cases}$$
and let $\tau_X^u=(\tau^u)_X$.

For an integer $k$, 
let $[k]$ denote the set of positive integers at most $k$, 
and let 
$${\cal P}_k(b)=\left\{ (b_1,\ldots,b_k)\in \mathbb{N}_0^k:b_1+\cdots+b_k=b\right\}$$
be the set of ordered partitions of $b$ into $k$ non-negative integers.

We devote separate subsections to the problems (\ref{ed1}) and (\ref{ed1b}).

\subsection{Calculating ${\rm vacc}_1(T,\tau,b)$}\label{subsec1}

For a vertex $u$ of $T$ and a non-negative integer $b$,
we consider the two values
\begin{itemize}
\item $x_0(u,b)={\rm vacc}_1(T_u,\tau,b)$, and 
\item $x_1(u,b)={\rm vacc}_1\left(T_u,\tau^u,b\right)$.
\end{itemize}
Intuitively, 
$x_1(u,b)$ corresponds to a situation,
where the infection reaches the parent of $u$ before it reaches $u$,
that is, the index $0$ or $1$ indicates the amount of help
that $u$ receives from outside of $T_u$.

If $b\leq n(T_u)$, 
then let $X_0(u,b)$ and $X_1(u,b)$ in ${V(T)\choose b}$ 
be such that 
\begin{eqnarray*}
x_0(u,b)&=&{\rm dyn}\left(T_u,\tau_{X_0(u,b)}\right)\mbox{ and }\\
x_1(u,b)&=&{\rm dyn}\left(T_u,\tau^u_{X_1(u,b)}\right),
\end{eqnarray*}
where, if possible, let $X_0(u,b)=X_1(u,b)$.

\begin{lemma}\label{lemma0}
$x_0(u,b)\geq x_1(u,b)$, and if $x_0(u,b)=x_1(u,b)$, then $X_0(u,b)=X_1(u,b)$.
\end{lemma}
\begin{proof}
If $x_1(u,b)=-\infty$, then the statement is trivial.
Hence, we may assume that $x_1(u,b)>-\infty$,
which implies that the set $X_1(u,b)$ is defined.
Let $D$ be a minimum dynamic monopoly of $\left(T_u,\tau_{X_1(u,b)}\right)$.
By the definition of ${\rm vacc}_1(T_u,\tau,b)$, we have $x_0(u,b)\geq |D|$.
Since $D$ is a dynamic monopoly of $\left(T_u,\tau_{X_1(u,b)}^u\right)$,
we obtain 
$x_0(u,b)\geq |D|\geq {\rm dyn}\left(T_u,\tau_{X_1(u,b)}^u\right)=x_1(u,b)$.
Furthermore, if $x_0(u,b)=x_1(u,b)$, then 
$x_0(u,b)=|D|={\rm dyn}\left(T_u,\tau_{X_1(u,b)}\right)$,
which implies $X_0(u,b)=X_1(u,b)$.
\end{proof}

\begin{lemma}\label{lemma1}
If $u$ is a leaf of $T$, then, for $j\in \{ 0,1\}$
\begin{eqnarray*}
x_j(u,b)&=&
\begin{cases}
0 &\mbox{, if $\tau(u)\leq j$ and $b=0$,}\\
-\infty &\mbox{, if $b\geq 2$, and}\\
1 &\mbox{, otherwise, and}
\end{cases}\\[5mm]
X_j(u,b) & = &  
\begin{cases}
\emptyset &\mbox{, if $b=0$, and}\\
\{ u\} &\mbox{, if $b=1$.}
\end{cases}
\end{eqnarray*}
\end{lemma}
\begin{proof}
These equalities follow immediately from the definitions.
\end{proof}
We have observed that 
$X_0(u,b) = X_1(u,b)$
if $x_0(u,b)=x_1(u,b)$ or $u$ is a leaf.
A surprising key insight that is essential for our approach 
is that $$X_0(u,b) = X_1(u,b)$$ 
always holds, 
which will follow by an inductive argument based on the following lemma.

\begin{lemma}\label{lemma2}
Let $u$ be a vertex of $T$ that is not a leaf, 
and let $b$ be a non-negative integer.

If $v_1,\ldots,v_k$ are the children of $u$,
and $X_0(v_i,b_i)=X_1(v_i,b_i)$ for every $i\in [k]$ and 
every non-negative integer $b_i$ at most $n(T_{v_i})$, 
then, for $j\in \{ 0,1\}$,
\begin{eqnarray}
x_j(u,b) &=& \max\Big\{ z(u,b),z_j(u,b)\Big\}\mbox{, and}\label{e1}\\
X_0(u,b) & = & X_1(u,b) \mbox{ if $b\leq n(T_u)$,}\label{e3}
\end{eqnarray}
where
\begin{eqnarray*}
z(u,b) &=& \max\left\{ 1+\sum\limits_{i=1}^kx_1(v_i,b_i):(b_1,\ldots,b_k)\in {\cal P}_k(b-1)\right\},\\
z_j(u,b) &=& \max\left\{ \delta_j(b_1,\ldots,b_k)+\sum\limits_{i=1}^kx_1(v_i,b_i):(b_1,\ldots,b_k)\in {\cal P}_k(b)\right\},
\end{eqnarray*}
and, for $(b_1,\ldots,b_k)\in {\cal P}_k(b)$,
\begin{eqnarray*}
\delta_j(b_1,\ldots,b_k)&=&
\begin{cases}
0 &\mbox{, if $\Big|\Big\{ i\in [k]:x_0(v_i,b_i)=x_1(v_i,b_i)\Big\}\Big|\geq \tau(u)-j$, and}\\
1 &\mbox{, otherwise.}
\end{cases}
\end{eqnarray*}
\end{lemma}
\begin{proof}
By symmetry, it suffices to consider the case $j=0$.

First, suppose that $b>n(T_u)$, which implies $x_0(u,b)=-\infty$.
Note that $n(T_u)=1+\sum\limits_{i=1}^kn(T_{v_i})$.
Therefore, 
if $(b_1,\ldots,b_k)\in {\cal P}_k(b-1)\cup {\cal P}_k(b)$,
then $b_i>n(T_{v_i})$ for some $i\in [k]$.
This implies $z(u,b)=z_0(u,b)=-\infty$,
and, hence, $\max\{ z(u,b),z_0(u,b)\}=-\infty=x_0(u,b)$.

Now, let $b\leq n(T_u)$, which implies $x_0(u,b)>-\infty$.
The following three claims complete the proof of (\ref{e1}).

\begin{claim}\label{claim1}
$x_0(u,b)\geq z(u,b)$.
\end{claim}
\begin{proof}[Proof of Claim \ref{claim1}]
It suffices to show 
$x_0(u,b)\geq 1+\sum\limits_{i=1}^kx_1(v_i,b_i)$ 
for every $(b_1,\ldots,b_k)$ in ${\cal P}_k(b-1)$
with $x_1(v_i,b_i)>-\infty$ for every $i\in [k]$.
Let $(b_1,\ldots,b_k)$ be such an element of ${\cal P}_k(b-1)$.

Since $x_1(v_i,b_i)>-\infty$, the set $X_1(v_i,b_i)$ is defined for every $i\in [k]$.

Let $X=\{ u\}\cup \bigcup\limits_{i=1}^kX_1(v_i,b_i)$.
Since $|X|=b$, we have $x_0(u,b)\geq {\rm dyn}(T_u,\tau_X)$.

Let $D$ be a minimum dynamic monopoly of $(T_u,\tau_X)$.
Since $u\in X$, we have $u\in D$.
For each $i\in [k]$, 
it follows that $D_i=D\cap V(T_{v_i})$ is a dynamic monopoly of 
$\left(T_{v_i},\tau_X^{v_i}\right)$.
Since, restricted to $V(T_{v_i})$, 
the two functions $\tau_X^{v_i}$
and $\tau_{X_1(v_i,b_i)}^{v_i}$ coincide,
we obtain $|D_i|\geq {\rm dyn}\left(T_{v_i},\tau_{X_1(v_i,b_i)}^{v_i}\right)=x_1(v_i,b_i)$.
Altogether, we conclude
$$x_0(u,b)\geq |D|
=1+\sum\limits_{i=1}^k|D_i|
\geq 1+\sum\limits_{i=1}^kx_1(v_i,b_i).$$
\end{proof}

\begin{claim}\label{claim2}
$x_0(u,b)\geq z_0(u,b)$.
\end{claim}
\begin{proof}[Proof of Claim \ref{claim2}]
Again, it suffices to show 
$x_0(u,b)\geq \delta_0(b_1,\ldots,b_k)+\sum\limits_{i=1}^kx_1(v_i,b_i)$ 
for every $(b_1,\ldots,b_k)$ in ${\cal P}_k(b)$
with $x_1(v_i,b_i)>-\infty$ for every $i\in [k]$.
Let $(b_1,\ldots,b_k)$ be such an element of ${\cal P}_k(b)$.

Since $x_1(v_i,b_i)>-\infty$, the set $X_1(v_i,b_i)$ is defined for every $i\in [k]$.

Let $X=\bigcup\limits_{i=1}^kX_1(v_i,b_i)$.
Since $|X|=b$, we have $x_0(u,b)\geq {\rm dyn}(T_u,\tau_X)$.

Let $D$ be a minimum dynamic monopoly of $(T_u,\tau_X)$,
that is, $|D|\leq x_0(u,b)$.
For each $i\in [k]$,
it follows that $D_i=D\cap V(T_{v_i})$ 
is a dynamic monopoly of $\left(T_{v_i},\tau_X^{v_i}\right)$,
which implies $|D_i|\geq x_1(v_i,b_i)$
as in the proof of Claim \ref{claim1}.

If $\delta_0(b_1,\ldots,b_k)=0$,
then $|D|\geq \sum\limits_{i=1}^k|D_i|\geq \delta_0(b_1,\ldots,b_k)+\sum\limits_{i=1}^kx_1(v_i,b_i)$.
Similarly, if $u\in D$, then 
$|D|=1+\sum\limits_{i=1}^k|D_i|
\geq \delta_0(b_1,\ldots,b_k)+\sum\limits_{i=1}^kx_1(v_i,b_i)$.
Therefore, we may assume that 
$\delta_0(b_1,\ldots,b_k)=1$ and that $u\not\in D$.
This implies that there is some $\ell\in [k]$ with $x_0(v_\ell,b_\ell)>x_1(v_\ell,b_\ell)$
such that $D_\ell$ is a dynamic monopoly of $\left(T_{v_\ell},\tau_X\right)$.
Since $X_0(v_\ell,b_\ell)=X_1(v_\ell,b_\ell)$, we obtain that, restricted to $V(T_{v_\ell})$,
the two functions $\tau_X$ and $\tau_{X_0(v_\ell,b_\ell)}$ coincide,
which implies
$|D_\ell|\geq 
{\rm dyn}\left(T_{v_\ell},\tau_{X_0(v_\ell,b_\ell)}\right)=x_0(v_\ell,b_\ell)\geq 1+x_1(v_\ell,b_\ell)$.
Therefore, also in this case,
$|D|=|D_\ell|+\sum\limits_{i\in [k]\setminus \{\ell\}}|D_i|
\geq \delta_0(b_1,\ldots,b_k)+\sum\limits_{i=1}^kx_1(v_i,b_i)$.
\end{proof}

\begin{claim}\label{claim3}
$x_0(u,b)\leq z(u,b)$ or $x_0(u,b)\leq z_0(u,b)$.
\end{claim}
\begin{proof}[Proof of Claim \ref{claim3}]
Let $X=X_0(u,b)$, that is, $x_0(u,b)={\rm dyn}(T_u,\tau_X)$.
Let $b_i=|X\cap V(T_{v_i})|$ for every $i\in [k]$.
Let $D_i$ be a minimum dynamic monopoly of 
$\left(T_{v_i},\tau^{v_i}_X\right)$
for every $i\in [k]$.
By the definition of $x_1(v_i,b_i)$, 
we obtain $|D_i|\leq x_1(v_i,b_i)$.
Let $D=\{ u\}\cup \bigcup\limits_{i=1}^kD_i$.
The set $D$ is a dynamic monopoly of $(T_u,\tau_X)$, which implies $x_0(u,b)\leq |D|$.

First, suppose that $u\in X$.
This implies $(b_1,\ldots,b_k)\in {\cal P}_k(b-1)$,
and, hence,
$$x_0(u,b)\leq |D|=1+\sum\limits_{i=1}^k|D_i|
\leq 1+\sum\limits_{i=1}^kx_1(v_i,b_i)\leq z(u,b).$$
Next, suppose that $u\not\in X$,
which implies $(b_1,\ldots,b_k)\in {\cal P}_k(b)$.

If $\delta_0(b_1,\ldots,b_k)=1$,
then 
$$x_0(u,b)
\leq |D|
=1+\sum\limits_{i=1}^k|D_i|
\leq \delta_0(b_1,\ldots,b_k)+\sum\limits_{i=1}^kx_1(v_i,b_i)
\leq z_0(u,b).$$
Therefore, we may assume that $\delta_0(b_1,\ldots,b_k)=0$.
By symmetry, we may assume that 
$x_0(v_i,b_i)=x_1(v_i,b_i)$ for every $i\in [\tau(u)]$.
Let $D_i'$ be a minimum dynamic monopoly of 
$\left(T_{v_i},\tau_X\right)$ for every $i\in [\tau(u)]$.
By the definition of $x_0(v_i,b_i)$, 
we obtain $|D'_i|\leq x_0(v_i,b_i)=x_1(v_i,b_i)$.
Let $D'=\bigcup\limits_{i\in [\tau(u)]}D'_i
\cup \bigcup\limits_{i\in [k]\setminus [\tau(u)]}D_i$.
The set $D'$ is a dynamic monopoly of $(T_u,\tau_X)$.
This implies
$$x_0(u,b)\leq |D'|
=\sum\limits_{i\in [\tau(u)]}|D'_i|
+\sum\limits_{i\in [k]\setminus [\tau(u)]}|D_i|
\leq\sum\limits_{i\in [k]}x_1(v_i,b_i)
\leq z_0(u,b),$$
which completes the proof of the claim.
\end{proof}
At this point, the proof of (\ref{e1}) is complete, and it remains to show (\ref{e3}).
If $x_0(u,b)=x_1(u,b)$, then (\ref{e3}) follows from Lemma \ref{lemma0}. Hence, we may assume that $x_0(u,b)>x_1(u,b)$.
Since, by definition, 
$$\delta_1(b_1,\ldots,b_k)\leq \delta_0(b_1,\ldots,b_k)\leq \delta_1(b_1,\ldots,b_k)+1$$ 
for every $(b_1,\ldots,b_k)\in {\cal P}_k(b)$,
we obtain $z_1(u,b)\leq z_0(u,b)\leq z_1(u,b)+1$.
Together with (\ref{e1}), 
the inequality $x_0(u,b)>x_1(u,b)$ implies that 
\begin{eqnarray*}
x_0(u,b)&=&z_0(u,b)>z_1(u,b)=x_1(u,b)\mbox{ and}\\ z_1(u,b)&=&z_0(u,b)-1.
\end{eqnarray*}
Let $(b_1,\ldots,b_k)\in {\cal P}_k(b)$ be such that 
$z_0(u,b)=
\delta_0(b_1,\ldots,b_k)+\sum\limits_{i=1}^kx_1(v_i,b_i)$.

We obtain
\begin{eqnarray*}
z_1(u,b)
&\geq &
\delta_1(b_1,\ldots,b_k)+\sum\limits_{i=1}^kx_1(v_i,b_i)\\
&\geq &\delta_0(b_1,\ldots,b_k)-1+\sum\limits_{i=1}^kx_1(v_i,b_i)\\
&=&z_0(u,b)-1\\
&=&z_1(u,b),
\end{eqnarray*}
which implies 
$z_1(u,b)=
\delta_1(b_1,\ldots,b_k)+\sum\limits_{i=1}^kx_1(v_i,b_i)$,
that is, the same choice of $(b_1,\ldots,b_k)$ in ${\cal P}_k(b)$
maximizes the terms defining $z_0(u,b)$ and $z_1(u,b)$.

Since $z_0(u,b)>z_1(u,b)$,
we obtain
$\delta_1(b_1,\ldots,b_k)=0$
and 
$\delta_0(b_1,\ldots,b_k)=1$,
which implies that there are exactly $\tau(u)-1$ indices $i$ in $[k]$
with $x_0(v_i,b_i)=x_1(v_i,b_i)$.
By symmetry, we may assume that 
$x_0(v_i,b_i)=x_1(v_i,b_i)$ for $i\in [\tau(u)-1]$
and 
$x_0(v_i,b_i)>x_1(v_i,b_i)$ for $i\in [k]\setminus 
[\tau(u)-1]$.

Let $X=\bigcup\limits_{i=1}^kX_0(v_i,b_i)$.
Note that, by assumption, we have $X=\bigcup\limits_{i=1}^kX_1(v_i,b_i)$.

Let $D$ be a minimum dynamic monopoly of $(T_u,\tau_X)$.
By the definition of $x_0(u,b)$, we have $|D|\leq x_0(u,b)$.
Let $D_i=D\cap V(T_{v_i})$ for every $i\in [k]$.
Since $D_i$ is a dynamic monopoly of $(T_{v_i},\tau^{v_i}_X)$ for every $i\in [k]$, we obtain $|D_i|\geq x_1(v_i,b_i)$.
Note that either $u\in D$ or $u\not\in D$ and there is some index $\ell\in [k]\setminus [\tau(u)-1]$ such that $D\cap V(T_{v_\ell})$ is a dynamic monopoly of $(T_{v_\ell},\tau_X)$.
In the first case,
we obtain 
$$z_0(u,b)=x_0(u,b)\geq |D|
=1+\sum\limits_{i=1}^k|D_i|
\geq 1+\sum\limits_{i=1}^kx_1(v_i,b_i)
=z_0(u,b),$$
and, in the second case, 
we obtain $|D_\ell|\geq x_0(v_\ell,b_\ell)\geq x_1(v_\ell,b_\ell)+1$, and, hence,
$$z_0(u,b)=x_0(u,b)\geq |D|
=|D_\ell|+\sum\limits_{i\in [k]\setminus \{ \ell\}}|D_i|
\geq 1+\sum\limits_{i=1}^kx_1(v_i,b_i)
=z_0(u,b).$$
In both cases we obtain $|D|=x_0(u,b)$,
which implies that $X_0(u,b)$ may be chosen equal to $X$.

Now, let $D^-$ be a minimum dynamic monopoly of $(T_u,\tau_X^u)$.
By the definition of $x_1(u,b)$, we have $|D^-|\leq x_1(u,b)$.
Let $D^-_i=D^-\cap V(T_{v_i})$ for every $i\in [k]$.
Since $D^-_i$ is a dynamic monopoly of $(T_{v_i},\tau^{v_i}_X)$ for every $i\in [k]$, we obtain $|D^-_i|\geq x_1(v_i,b_i)$.
Now,
$$z_1(u,b)=x_1(u,b)
\geq |D^-|
\geq \sum\limits_{i=1}^kx_1(v_i,b_i)
=z_1(u,b),$$
which implies that $|D^-|=x_1(u,b)$,
and that $X_1(u,b)$ may be chosen equal to $X$.
Altogether, 
the two sets $X_0(u,b)$ and $X_1(u,b)$ may be chosen equal, 
which implies (\ref{e3}).
\end{proof}
Applying induction 
using Lemma \ref{lemma1} und Lemma \ref{lemma2}, 
we obtain the following.

\begin{corollary}\label{corollary1}
$X_0(u,b)=X_1(u,b)$ for every vertex $u$ of $T$, and every non-negative integer $b$ at most $n(T_u)$.
\end{corollary}
Apart from the specific values of $x_0(u,b)$ and $x_1(u,b)$,
the arguments in the proof of Lemma \ref{lemma2} 
also yield feasible recursive choices for $X_0(u,b)$.
In fact, 
if 
$x_0(u,b)=z(u,b)=1+\sum\limits_{i=1}^kx_1(v_i,b_i)$
for some $(b_1,\ldots,b_k)\in {\cal P}_k(b-1)$,
then 
$$X_0(u,b)=\{ u\}\cup \bigcup\limits_{i=1}^kX_0(v_i,b_i)$$
is a feasible choice, and
if 
$x_0(u,b)=z_0(u,b)=\delta(b_1,\ldots,b_k)+\sum\limits_{i=1}^kx_1(v_i,b_i)$
for some $(b_1,\ldots,b_k)\in {\cal P}_k(b)$,
then 
$$X_0(u,b)=\bigcup\limits_{i=1}^kX_0(v_i,b_i)$$
is a feasible choice.
While the expressions in Lemma \ref{lemma2}
involve the maximization over the elements of ${\cal P}_k(b-1)$ and ${\cal P}_k(b)$, which may be exponentially large,
we now show that the values $x_0(u,b)$ and $x_1(u,b)$
can be computed efficiently.

\begin{lemma}\label{lemma3}
Let $u$ be a vertex of $T$ that is not a leaf, 
let $b$ be a non-negative integer,
and let $v_1,\ldots,v_k$ be the children of $u$.

If the values $x_1(v_i,b_i)$ are given for every $i\in [k]$ and every non-negative integer $b_i$ at most $n(T_{v_i})$,
then $x_0(u,b)$ and $x_1(u,b)$ can be computed in $O(k^2b)$ time.
\end{lemma}
\begin{proof}
By symmetry, it suffices to consider $x_0(u,b)$. 
We explain how to efficiently compute $z_0(u,b)$;
a simplified approach works for $z(u,b)$.

For $p\in \{0\}\cup [k]$, 
an integer $p_=$, 
and $b'\in \{0\}\cup [b]$, 
let $M(p,p_=,b')$ be defined as the maximum of the expression
$\sum\limits_{i=1}^px_1(v_i,b_i)$
maximized over all
$(b_1,\ldots,b_p)\in {\cal P}_p(b')$
with 
$$p_==\Big|\Big\{ i\in [p]:x_0(v_i,b_i)=x_1(v_i,b_i)\Big\}\Big|.$$
Clearly, $M(p,p_=,b')=-\infty$ if either $p<p_=$ or $p_=<0$, 
and $M(0,0,b')=0$.

For $p\in [k]$, the value of $M(p,p_=,b')$ 
is the maximum of the following two values:
\begin{itemize}
\item The maximum of 
$$M(p-1,p_=-1,b_{\leq p-1})+x_1(v_p,b_p)$$
over all $(b_{\leq p-1},b_p)\in {\cal P}_2(b')$
with $x_0(v_p,b_p)=x_1(v_p,b_p)$, and
\item the maximum of 
$$M(p-1,p_=,b_{\leq p-1})+x_1(v_p,b_p)$$
over all $(b_{\leq p-1},b_p)\in {\cal P}_2(b')$
with $x_0(v_p,b_p)>x_1(v_p,b_p)$,
\end{itemize}
which implies that $M(p,p_=,b')$ can be determined in $O(b')$ time
given the values 
$$\mbox{$M(p-1,p_=,b_{\leq p-1})$, $M(p-1,p_=-1,b_{\leq p-1})$, $x_0(v_p,b_p)$, and $x_1(v_p,b_p)$.}$$
Altogether, the values $M(k,p_=,b)$ 
for all $p_=\in \{ 0\}\cup [k]$ can be determined in $O(k^2b)$ time.
By the definition of $\delta_0(b_1,\ldots,b_k)$,
the value of $z_0(u,b)$ equals the maximum of the two expressions
$$1+\max\Big\{ M(k,p_=,b):p_=\in \{ 0\}\cup [\tau(u)-1]\Big\}$$
and
$$\max\Big\{ M(k,p_=,b):p_=\in [k]\setminus [\tau(u)-1]\Big\},$$
which completes the proof.
\end{proof}
We proceed to our first main result.

\begin{theorem}\label{theorem1}
For a given triple $(T,\tau,b)$,
where $T$ is a tree of order $n$,
$\tau$ is a threshold function for $T$,
and $b$ is a non-negative integer at most $n$,
the value of ${\rm vacc}_1(T,\tau,b)$
as well as a set $X$ in ${V(T)\choose b}$ 
with ${\rm vacc}_1(T,\tau,b)={\rm dyn}\left(T,\tau_X\right)$
can be determined in $O(n^2b^2)$ time.
\end{theorem}
\begin{proof}
Given $(T,\tau,b)$,
Lemma \ref{lemma1} and Lemma \ref{lemma3} imply
that the values of $x_0(u,b')$ and of $x_1(u,b')$ 
for all $u\in V(T)$ and all $b'\in \{ 0\}\cup [b]$
can be determined in time
$O\Big(\sum\limits_{u\in V(T)}d_T(u)^2b^2\Big)$.
It is a simple folklore exercise that $\sum\limits_{u\in V(T)}d_T(u)^2\leq n^2-n$
for every tree $T$ of order $n$, which implies the statement about the running time.
Since ${\rm vacc}_1(T,\tau,b)=x_0(r,b)$, 
the statement about the value of ${\rm vacc}_1(T,\tau,b)$ follows.
The statement about the set $X$
follows easily from the remark after Corollary \ref{corollary1} 
concerning the sets $X_0(u,b)$, and the proof of Lemma \ref{lemma3},
where, next to the values $M(p,p_=,b')$, one may also memorize suitable maximizers.
\end{proof}

\subsection{Calculating ${\rm vacc}_2(T,\tau,b)$}\label{subsec2}

Our approach for ${\rm vacc}_2(T,\tau,b)$ is similar to the one for ${\rm vacc}_1(T,\tau,b)$ with some additional complications.

For a vertex $u$ of $T$ and a non-negative integer $b$,
we consider the three values
\begin{itemize}
\item 
$y_\in(u,b)
=\max\left\{ {\rm dyn}(T_u-Y,\tau):Y\in {V(T_u)\choose b}\mbox{ with }u\in Y\right\},$
\item $y_0(u,b)
=\max\left\{ {\rm dyn}(T_u-Y,\tau):Y\in {V(T_u)\choose b}\mbox{ with }u\not\in Y\right\},$ and
\item $y_1(u,b)
=\max\left\{ {\rm dyn}(T_u-Y,\tau^u):Y\in {V(T_u)\choose b}\mbox{ with }u\not\in Y\right\}.$
\end{itemize}
Clearly, 
\begin{eqnarray}\label{ev2}
{\rm vacc}_2(T_u,\tau,b)=\max\Big\{ y_\in(u,b),y_0(u,b)\Big\}.
\end{eqnarray}
Note that 
$y_\in(u,b)=-\infty$ if $b>n(T_u)$,
and that 
$y_j(u,b)=-\infty$ if $b>n(T_u)-1$ for $j\in \{ 0,1\}$.
If $b\leq n(T_u)$, 
then let $Y_\in(u,b)$ in ${V(T)\choose b}$ with $u\in Y_\in(u,b)$
be such that 
\begin{eqnarray*}
y_\in(u,b)&=&{\rm dyn}\left(T_u-Y_\in(u,b),\tau\right).
\end{eqnarray*}
Similarly, if $b\leq n(T_u)-1$, 
then let $Y_0(u,b)$ and $Y_1(u,b)$ in ${V(T)\choose b}$ 
with $u\not\in Y_0(u,b)$ and $u\not\in Y_1(u,b)$
be such that 
\begin{eqnarray*}
y_0(u,b)&=&{\rm dyn}\left(T_u-Y_0(u,b),\tau\right)\mbox{, and}\\
y_1(u,b)&=&{\rm dyn}\left(T_u-Y_1(u,b),\tau^u\right),
\end{eqnarray*}
where, if possible, let $Y_0(u,b)=Y_1(u,b)$;
again, it will be a key insight that the last equality always holds.

The next lemma can be shown exactly as Lemma \ref{lemma0}.

\begin{lemma}\label{lemma0b}
$y_0(u,b)\geq y_1(u,b)$, and if $y_0(u,b)=y_1(u,b)$, then $Y_0(u,b)=Y_1(u,b)$.
\end{lemma}
The next lemma corresponds to Lemma \ref{lemma1}.

\begin{lemma}\label{lemma1b}
If $u$ is a leaf of $T$, then, for $j\in \{ 0,1\}$,
\begin{eqnarray*}
y_\in(u,1)&=& 0,\\
Y_\in(u,1)&=& \{ u\},\\
y_j(u,0)&=&
\begin{cases}
0 &\mbox{, if $\tau(u)\leq j$,}\\
1 & \mbox{, otherwise, and}
\end{cases}\\
Y_j(u,0) & = & \emptyset.
\end{eqnarray*}
\end{lemma}
\begin{proof}
These equalities follow immediately from the definitions.
\end{proof}
The following two lemmas correspond to Lemma \ref{lemma2}.

\begin{lemma}\label{lemma2b1}
Let $u$ be a vertex of $T$ that is not a leaf, 
and let $b$ be a non-negative integer.

If $v_1,\ldots,v_k$ are the children of $u$, then
\begin{eqnarray*}
y_\in(u,b) &=& 
\max\left\{ \sum\limits_{i=1}^k
\max\Big\{ y_\in(v_i,b_i),y_0(v_i,b_i)\Big\}:
(b_1,\ldots,b_k)\in {\cal P}_k(b-1)\right\}.
\end{eqnarray*}
\end{lemma}
\begin{proof}
Since $u$ is removed from $T_u$, 
this follows immediately from (\ref{ev2}).
\end{proof}
Clearly, if 
$y_\in(u,b)= \sum\limits_{i=1}^k
\max\Big\{ y_\in(v_i,b_i),y_0(v_i,b_i)\Big\}$
for some $(b_1,\ldots,b_k)\in {\cal P}_k(b-1)$, then
$$Y_\in(u,b)=\{ u\}
\cup\bigcup\limits_{i\in [k]:y_\in(v_i,b_i)>y_0(v_i,b_i)}Y_\in(v_i,b_i)
\cup\bigcup\limits_{i\in [k]:y_\in(v_i,b_i)\leq y_0(v_i,b_i)}Y_0(v_i,b_i)
$$
is a feasible recursive choice for $Y_\in(u,b)$.

\begin{lemma}\label{lemma2b2}
Let $u$ be a vertex of $T$ that is not a leaf, 
and let $b$ be a non-negative integer.

If $v_1,\ldots,v_k$ are the children of $u$,
and $Y_0(v_i,b_i)=Y_1(v_i,b_i)$ for every $i\in [k]$ and 
every non-negative integer $b_i$ at most $n(T_{v_i})-1$, 
then, for $j\in \{ 0,1\}$,
\begin{eqnarray}
y_j(u,b)&=&z_j' (u,b), \mbox{ and } \label{e2222}\\
Y_0(u,b) & = & Y_1(u,b) \mbox{ if $b\leq n(T_u)-1$,}\label{e3b}
\end{eqnarray}
where
\begin{eqnarray}
z_j'(u,b) &=& \max\Bigg\{ 
\delta_j\left(\vec{b},\vec{c}\right)
+\sum\limits_{i\in [k]:c_i=1}y_\in(v_i,b_i)
+\sum\limits_{i\in [k]:c_i=0}y_1(v_i,b_i):\nonumber\\[3mm]
&& \,\,\,\,\,\,\,\,\,\,\,\,\,\,\,\,\, 
\vec{b}=(b_1,\ldots,b_k)\in {\cal P}_k(b)
\mbox{ and }
\vec{c}=(c_1,\ldots,c_k)\in \{ 0,1\}^k\mbox{ with $\vec{b}\geq \vec{c}$ }\Bigg\}\label{e1b},
\end{eqnarray}
and, for $\vec{b}$ and $\vec{c}$ as in (\ref{e1b}),
$$
\delta_j\left(\vec{b},\vec{c}\right)=
\begin{cases}
0 &\mbox{, if $\Big|\Big\{ i\in [k]:c_i=0\mbox{ and }y_0(v_i,b_i)=y_1(v_i,b_i)\Big\}\Big|\geq \tau(u)-j$, and}\\
1 &\mbox{, otherwise.}
\end{cases}
$$
\end{lemma}
\begin{proof}
By symmetry, it suffices to consider the case $j=0$.

First, suppose that $b\geq n(T_u)$, which implies $y_0(u,b)=-\infty$.
If $(b_1,\ldots,b_k)\in {\cal P}_k(b)$,
then $b_i>n(T_{v_i})$ for some $i\in [k]$,
and, hence, $z_0'(u,b)=-\infty=y_0(u,b)$. 

Now, let $b\leq n(T_u)-1$. The following two claims complete the proof of (\ref{e2222}). 

\setcounter{claim}{0}

\begin{claim}\label{claim4}
$y_0(u,b) \geq z_0'(u,b)$.
\end{claim}

\begin{proof}[Proof of Claim \ref{claim4}]
It suffices to show 
$y_0(u,b)\geq \delta_0\left(\vec{b},\vec{c}\right)
+\sum\limits_{i\in [k]:c_i=1}y_\in(v_i,b_i)
+\sum\limits_{i\in [k]:c_i=0}y_1(v_i,b_i)$
for every $\vec{b}=(b_1,\ldots,b_k)\in {\cal P}_k(b)$
and
$\vec{c}=(c_1,\ldots,c_k)\in \{ 0,1\}^k$
with $\vec{b}\geq \vec{c}$ such that for every $i\in[k]$, 
if $c_i=1$, then $y_\in(v_i,b_i)>-\infty$,
and if $c_i=0$, then $y_1(v_i,b_i)>-\infty$.
Let $\vec{b}$ and $\vec{c}$ be such candidates.
Since $y_\in(v_i,b_i)>-\infty$ and $y_1(v_i,b_i)>-\infty$, 
the sets $Y_\in(v_i,b_i)$ and $Y_1(v_i,b_i)$ are defined.
Let
$Y=\bigcup\limits_{i\in [k]: c_i=1}Y_\in(v_i,b_i)
\cup \bigcup\limits_{i\in [k]: c_i=0}Y_1(v_i,b_i)$.
Since $|Y|=b$ and $u\not\in Y$, we have $y_0(u,b)\geq {\rm dyn}\left(T_u-Y,\tau\right)$.

Let $D$ be a minimum dynamic monopoly of $\left(T_u-Y,\tau\right)$, that is, $|D|\leq y_0(u,b)$. 
For every $i\in[k]$, it follows that $D_i=D\cap V(T_{v_i})$ is a dynamic monopoly of $\left(T_{v_i}-Y,\tau^{v_i}\right)$, which
implies that 
$$|D_i|\geq {\rm dyn}\left(T_{v_i}-Y,\tau^{v_i}\right)
=
\begin{cases}
	y_\in(v_i,b_i)&\mbox{, if } c_i=1,\mbox{ and} \\
	y_1(v_i,b_i)&\mbox{, if } c_i=0.
\end{cases}
$$
If $\delta_0\left(\vec{b},\vec{c}\right)=0$, then $$|D|\geq \sum\limits_{i=1}^k|D_i|\geq \delta_0\left(\vec{b},\vec{c}\right)
+\sum\limits_{i\in [k]:c_i=1}y_\in(v_i,b_i)
+\sum\limits_{i\in [k]:c_i=0}y_1(v_i,b_i).$$
Similarly, if $u\in D$, then 
$$|D|= 1+\sum\limits_{i=1}^k|D_i|
\geq \delta_0\left(\vec{b},\vec{c}\right)
+\sum\limits_{i\in [k]:c_i=1}y_\in(v_i,b_i)
+\sum\limits_{i\in [k]:c_i=0}y_1(v_i,b_i).$$
Hence, we may assume that $\delta_0\left(\vec{b},\vec{c}\right)=1$ and that $u\not\in D$.
This implies that there is some $\ell\in [k]$ with $c_\ell=0$ and $y_0(v_\ell,b_\ell)>y_1(v_\ell,b_\ell)$
such that $D_\ell$ is a dynamic monopoly of $\left(T_{v_\ell}-Y,\tau\right)
=\left(T_{v_\ell}-Y_1(v_\ell,b_\ell),\tau\right)$.
By assumption, we have $Y_0(v_\ell,b_\ell)=Y_1(v_\ell,b_\ell)$, and hence 
$$|D_\ell|
\geq {\rm dyn}\left(T_{v_\ell}-Y_0(v_\ell,b_\ell),\tau\right)
= y_0(v_\ell,b_\ell)
\geq 1+y_1(v_\ell,b_\ell).$$
Therefore, also in this case, 
$$|D|=
|D_\ell|+\sum\limits_{i\in[k]\setminus\{\ell\}}|D_i|
\geq \delta_0\left(\vec{b},\vec{c}\right)
+\sum\limits_{i\in [k]:c_i=1}y_\in(v_i,b_i)
+\sum\limits_{i\in [k]:c_i=0}y_1(v_i,b_i).$$
\end{proof}

\begin{claim}\label{claim5}
$y_0(u,b) \leq z_0'(u,b)$.
\end{claim}

\begin{proof}[Proof of Claim \ref{claim5}]
Let $Y=Y_0(u,b)$, that is, $y_0(u,b)={\rm dyn}(T_u-Y,\tau)$.
For every $i\in[k]$, 
let $b_i=|Y\cap V(T_{v_i})|$, and 
let $c_i=1$ if $v_i\in Y$, and $c_i=0$ if $v_i\not\in Y$.
Let $D_i$ be a minimum dynamic monopoly of $\left(T_{v_i}-Y,\tau^{v_i}\right)$ for every $i\in[k]$.
By the definition of $y_\in(v_i,b_i)$ and $y_1(v_i,b_i)$, we obtain 
$$|D_i|\leq
\begin{cases}
y_\in(v_i,b_i)&\mbox{, if } c_i=1,\mbox{ and} \\
y_1(v_i,b_i)&\mbox{, if } c_i=0.
\end{cases}$$
Let $D=\{ u\}\cup \bigcup\limits_{i=1}^kD_i$. The set $D$ is a dynamic monopoly of $\left(T_u-Y,\tau\right)$, which implies $y_0(u,b)\leq|D|$.
By the definition of $Y_0(u,b)$, we have $u\not\in Y$, 
which implies $(b_1,\ldots,b_k)\in {\cal P}_k(b)$ and 
$\vec{b}\geq \vec{c}$.

If $\delta_0\left(\vec{b},\vec{c}\right)=1$, then
$$y_0(u,b)
\leq |D|
= 1+\sum\limits_{i=1}^k|D_i|
\leq \delta_0\left(\vec{b},\vec{c}\right)+
\sum\limits_{i\in[k]:c_i=1}y_\in(v_i,b_i)+
\sum\limits_{i\in[k]:c_i=0}y_1(v_i,b_i)
\leq z_0'(u,b).$$
Therefore, we may assume that $\delta_0\left(\vec{b},\vec{c}\right)=0$.
By symmetry, we may assume that $c_i=0$ and $y_0(v_i,b_i)=y_1(v_i,b_i)$ for every $i\in[\tau(u)]$.
Let $D_i'$ be a minimum dynamic monopoly of $\left(T_{v_i}-Y,\tau\right)$ for every $i\in[\tau(u)]$.
By the definition of $y_0(v_i,b_i)$, we obtain $|D_i'|\leq y_0(v_i,b_i)=y_1(v_i,b_i)$.
Let $D'=\bigcup\limits_{i\in [\tau(u)]}D'_i
\cup \bigcup\limits_{i\in [k]\setminus [\tau(u)]}D_i$.
The set $D'$ is a dynamic monopoly of $(T_u-Y,\tau)$.
This implies
$$y_0(u,b)\leq |D'|
=\sum\limits_{i\in [\tau(u)]}|D'_i|
+\sum\limits_{i\in [k]\setminus [\tau(u)]}|D_i|
\leq\sum\limits_{i\in[k]:c_i=1}y_\in(v_i,b_i)+
\sum\limits_{i\in[k]:c_i=0}y_1(v_i,b_i)
\leq z_0'(u,b),$$
which completes the proof of the claim.
\end{proof}
It remains to show (\ref{e3b}). 
If $y_0(u,b)=y_1(u,b)$, then (\ref{e3b}) follows from Lemma~\ref{lemma0b}.
Hence, we may assume that $y_0(u,b)>y_1(u,b)$. 
Since, by definition,
$$\delta_1\left(\vec{b},\vec{c}\right)\leq \delta_0\left(\vec{b},\vec{c}\right)\leq \delta_1\left(\vec{b},\vec{c}\right)+1,$$
for every 
$\vec{b}=(b_1,\ldots,b_k)\in {\cal P}_k(b)$
and
$\vec{c}=(c_1,\ldots,c_k)\in \{ 0,1\}^k$  with $\vec{b}\geq \vec{c}$, 
we obtain $z_1'(u,b)\leq z_0'(u,b)\leq z_1'(u,b)+1$.
Together with (\ref{e2222}), this implies that
\begin{eqnarray*}
	y_0(u,b)&=&z_0'(u,b)>z_1'(u,b)=y_1(u,b)\mbox{ and }\\ z_1'(u,b)&=&z_0'(u,b)-1.
\end{eqnarray*} 
Let $\vec{b}$ and $\vec{c}$ be defined as above and such that $$z_0'(u,b)=\delta_0\left(\vec{b},\vec{c}\right)
+\sum\limits_{i\in [k]:c_i=1}y_\in(v_i,b_i)
+\sum\limits_{i\in [k]:c_i=0}y_1(v_i,b_i).$$

We obtain
\begin{eqnarray*}
	z_1'(u,b)
	&\geq &
\delta_1\left(\vec{b},\vec{c}\right)
+\sum\limits_{i\in [k]:c_i=1}y_\in(v_i,b_i)
+\sum\limits_{i\in [k]:c_i=0}y_1(v_i,b_i)\\
	&\geq &\delta_0\left(\vec{b},\vec{c}\right)-1
	+\sum\limits_{i\in [k]:c_i=1}y_\in(v_i,b_i)
	+\sum\limits_{i\in [k]:c_i=0}y_1(v_i,b_i)\\
	&=&z_0'(u,b)-1\\
	&=&z_1'(u,b),
\end{eqnarray*}
which implies $z_1'(u,b)=\delta_1\left(\vec{b},\vec{c}\right)
+\sum\limits_{i\in [k]:c_i=1}y_\in(v_i,b_i)
+\sum\limits_{i\in [k]:c_i=0}y_1(v_i,b_i)$.

Since $z_0'(u,b)>z_1'(u,b)$, we obtain $\delta_1\left(\vec{b},\vec{c}\right)=0$ and $\delta_0\left(\vec{b},\vec{c}\right)=1$,
which implies that there are exactly $\tau(u)-1$ indices $i\in[k]$ with $c_i=0$ and $y_0(v_i,b_i)=y_1(v_i,b_i)$.
By symmetry, we may assume that $c_i=0$ and $y_0(v_i,b_i)=y_1(v_i,b_i)$ for every $i\in[\tau(u)-1]$ and $y_0(v_i,b_i)>y_1(v_i,b_i)$ for every $i\in[k]\setminus[\tau(u)-1]$ with $c_i=0$.

Let 
$Y=\bigcup\limits_{i\in [k]: c_i=1}Y_\in(v_i,b_i)
\cup \bigcup\limits_{i\in [k]: c_i=0}Y_0(v_i,b_i)$. Note that by assumption, we have $Y_0(v_i,b_i)=Y_1(v_i,b_i)$ for every $i\in[k]$ with $c_i=0$.

Let $D$ be a minimum dynamic monopoly of $(T_u-Y,\tau)$.
By the definition of $y_0(u,b)$, we have $|D|\leq y_0(u,b)$.
Let $D_i=D\cap V(T_{v_i})$ for every $i\in[k]$.
Since $D_i$ is a dynamic monopoly of $\left(T_{v_i}-Y,\tau^{v_i}\right)$,
we obtain that 
$$|D_i|
\geq {\rm dyn}\left(T_{v_i}-Y,\tau^{v_i}\right)
=
\begin{cases}
y_\in(v_i,b_i)&\mbox{, if } c_i=1,\mbox{ and} \\
y_1(v_i,b_i)&\mbox{, if } c_i=0.
\end{cases}$$
Note that either $u\in D$, or $u\not\in D$ and there is some index $\ell\in[k]\setminus[\tau(u)-1]$ with $c_\ell=0$ such that $D_\ell$ is a dynamic monopoly of $(T_{v_\ell}-Y,\tau)$.

In the first case, we obtain
$$z_0'(u,b)
=y_0(u,b)
\geq |D|
=1+\sum\limits_{i=1}^k|D_i|
\geq 1+\sum\limits_{i\in[k]:c_i=1}y_\in(v_i,b_i)+
\sum\limits_{i\in[k]:c_i=0}y_1(v_i,b_i)
=z_0'(u,b),$$
and, in the second case, we obtain $|D_\ell|\geq y_0(v_\ell,b_\ell)\geq y_1(v_\ell,b_\ell)+1$, and, hence,
$$z_0'(u,b)
\geq |D|
=|D_\ell|+\sum\limits_{i\in[k]\setminus\{\ell\}}|D_i|
\geq 1+\sum\limits_{i\in[k]:c_i=1}y_\in(v_i,b_i)+
\sum\limits_{i\in[k]:c_i=0}y_1(v_i,b_i)
=z_0'(u,b).$$
In both cases we obtain $|D|=y_0(u,b)$, which implies that $Y_0(u,b)$ may be chosen equal to $Y$.

Now, let $D^-$ be a minimum dynamic monopoly of $\left(T_u-Y,\tau^u\right)$. 
By the definition of $y_1(u,b)$, we have $|D^-|\leq y_1(u,b)$. 
Let $D_i^-=D^-\cap V(T_{v_i})$ for every $i\in[k]$. 
Since $D_i^-$ is a dynamic monopoly of $\left(T_{v_i}-Y,\tau^{v_i}\right)$ for every $i\in[k]$, we obtain as above
$$|D_i^-|\geq 
\begin{cases}
y_\in(v_i,b_i)&\mbox{, if } c_i=1,\mbox{ and} \\
y_1(v_i,b_i)&\mbox{, if } c_i=0.
\end{cases}$$
Now,
$$z_1'(u,b)
=y_1(u,b)
\geq |D^-|
\geq \sum\limits_{i\in[k]:c_i=1}y_\in(v_i,b_i)+
\sum\limits_{i\in[k]:c_i=0}y_1(v_i,b_i)
=z_1'(u,b),$$
which implies that $|D^-|=Y_1(u,b)$, and that $Y_1(u,b)$ may be chosen equal to $Y$. Altogether, the two sets $Y_0(u,b)$ and $Y_1(u,b)$ may be chosen equal, which implies (\ref{e3b}).
\end{proof}
Applying induction 
using Lemma \ref{lemma1b} und Lemma \ref{lemma2b2}, 
we obtain the following.

\begin{corollary}\label{corollary1b}
$Y_0(u,b)=Y_1(u,b)$ for every vertex $u$ of $T$, and every non-negative integer $b$ at most $n(T_u)-1$.
\end{corollary}
Similarly as above, 
if 
$$y_0(u,b) =
\delta_0\left(\vec{b},\vec{c}\right)
+\sum\limits_{i\in [k]:c_i=1}y_\in(v_i,b_i)
+\sum\limits_{i\in [k]:c_i=0}y_1(v_i,b_i)$$
for some $\vec{b}$ and $\vec{c}$ as in (\ref{e1b}),
then
$$Y_0(u,b) =
\bigcup\limits_{i\in [k]:c_i=1}Y_\in(v_i,b_i)
\cup\bigcup\limits_{i\in [k]:c_i=0}Y_0(v_i,b_i)$$
is a feasible recursive choice for $Y_0(u,b)$.

The next lemma corresponds to Lemma \ref{lemma3}.

\begin{lemma}\label{lemma3b}
Let $u$ be a vertex of $T$ that is not a leaf, 
let $b$ be a non-negative integer,
and let $v_1,\ldots,v_k$ be the children of $u$.

If the values $y_\in(v_i,b_i)$ 
for every $i\in [k]$ and every non-negative integer $b_i$ 
at most $n(T_{v_i})$
as well as 
the values $y_0(v_i,b_i)$ and $y_1(v_i,b_i)$
for every $i\in [k]$ and every non-negative integer $b_i$ 
at most $n(T_{v_i})-1$
are given, 
then 
$y_\in(u,b)$,
$y_0(u,b)$, and
$y_1(u,b)$
can be computed in $O(k^3b)$ time.
\end{lemma}
\begin{proof}
Standard dynamic programming based on Lemma \ref{lemma2b1}
immediately implies the statement for $y_\in(u,b)$.
By symmetry, 
it suffices to explain how to efficiently compute $y_0(u,b)$.
For $p\in \{0\}\cup [k]$, 
an integer $p_\in$, 
an integer $p_=$, 
and $b'\in \{0\}\cup [b]$, 
let $M(p,p_\in,p_=,b')$ be defined as the maximum of the expression
$$\sum\limits_{i\in [p]:c_i=1}y_\in(v_i,b_i)
+\sum\limits_{i\in [p]:c_i=0}y_1(v_i,b_i)$$
maximized over all
$\vec{b}=(b_1,\ldots,b_p)\in {\cal P}_k(b')$
and
$\vec{c}=(c_1,\ldots,c_p)\in \{ 0,1\}^p$
with $\vec{b}\geq \vec{b}$ such that 
\begin{eqnarray*}
p_\in&=&c_1+\ldots+c_p\mbox{ and}\\
p_=&=&\Big|\Big\{ i\in [p]:c_i=0\mbox{ and }y_0(v_i,b_i)=y_1(v_i,b_i)\Big\}\Big|.
\end{eqnarray*}
As usual, if no such pair $\vec{b}$ and $\vec{c}$ exists, 
then $M(p,p_\in,p_=,b')=-\infty$.

Clearly, $M(0,0,0,b')=0$.

For $p\in [k]$, the value of $M(p,p_\in,p_=,b')$ 
is the maximum of the following three values:
\begin{itemize}
\item The maximum of 
$$M(p-1,p_\in-1,p_=,b_{\leq p-1})+y_\in(v_p,b_p)$$
over all $(b_{\leq p-1},b_p)\in {\cal P}_2(b')$,
\item the maximum of 
$$M(p-1,p_\in,p_=-1,b_{\leq p-1})+y_1(v_p,b_p)$$
over all $(b_{\leq p-1},b_p)\in {\cal P}_2(b')$
with $y_0(v_p,b_p)=y_1(v_p,b_p)$, and
\item the maximum of 
$$M(p-1,p_\in,p_=,b_{\leq p-1})+y_1(v_p,b_p)$$
over all $(b_{\leq p-1},b_p)\in {\cal P}_2(b')$
with $y_0(v_p,b_p)>y_1(v_p,b_p)$.
\end{itemize}
This implies that the values 
$M(k,p_\in,p_=,b)$ for all $p_\in,p_=\in \{ 0\}\cup [k]$ 
can be determined in $O(k^3b)$ time.
By the definition of $\delta_j\left(\vec{b},\vec{c}\right)$,
the value of $y_0(u,b)$ equals the maximum of the two expressions
$$1+\max\Big\{ M(k,p_\in,p_=,b):
p_\in\in \{ 0\}\cup [k]\mbox{ and }
p_=\in \{ 0\}\cup [\tau(u)-1]\Big\}$$
and
$$\max\Big\{ M(k,p_\in,p_=,b):
p_\in\in \{ 0\}\cup [k]\mbox{ and }
p_=\in [k]\setminus [\tau(u)-1]\Big\},$$
which completes the proof.
\end{proof}
We proceed to our second main result.

\begin{theorem}\label{theorem1b}
For a given triple $(T,\tau,b)$,
where $T$ is a tree of order $n$,
$\tau$ is a threshold function for $T$,
and $b$ is a non-negative integer at most $n$,
the value of ${\rm vacc}_2(T,\tau,b)$
as well as a set $Y$ in ${V(T)\choose b}$ 
with ${\rm vacc}_2(T,\tau,b)={\rm dyn}\left(T-Y,\tau\right)$
can be determined in $O(n^3b^2)$ time.
\end{theorem}
\begin{proof}
Given $(T,\tau,b)$,
Lemma \ref{lemma1b} and Lemma \ref{lemma3b} imply
that the values 
$y_\in(u,b')$,
$y_0(u,b')$, and
$y_1(u,b')$
for all $u\in V(T)$ and all $b'\in \{ 0\}\cup [b]$
can be computed in time
$O\Big(\sum\limits_{u\in V(T)}d_T(u)^3b^2\Big)$.
It is a simple folklore exercise that $\sum\limits_{u\in V(T)}d_T(u)^3\leq (n-1)\left((n-1)^2+1\right)$
for every tree $T$ of order $n$, 
which implies the statement about the running time.
The statement about the value of ${\rm vacc}_2(T,\tau,b)$ follows
from (\ref{ev2}). 
Finally, the statement about the set $Y$
follows easily from the remarks after 
Lemma \ref{lemma2b1}
and 
Corollary \ref{corollary1b}.
\end{proof}

\end{document}